\documentclass[11pt]{article}
\usepackage[english]{babel}
\usepackage{amsmath,amsthm,amssymb}
\usepackage[all]{xy}
\usepackage{setspace}
\usepackage[top=1.2in, bottom=1.4in, left=1.3in, right=1.3in]{geometry}
\usepackage{color}
\definecolor{grau}{rgb}{0.5,0.5,0.5}
\usepackage[colorlinks, linkcolor=grau, citecolor=grau, urlcolor=grau]{hyperref}
\usepackage{breakurl}
\usepackage{bibspacing}
\usepackage{leftidx}
\newtheorem{theorem}{Theorem}[section]
\newtheorem{lemma}[theorem]{Lemma}
\newtheorem{proposition}[theorem]{Proposition}

\numberwithin{equation}{section}

\def\sp{\textnormal{Spec}}
\def\P1k{\mathbb P_{k}^{1}}
\def\deg{\textnormal{deg} \ }
\def\slos{\textnormal{SL}_2(\mathcal O_{\Sb})}
\def\glos{\textnormal{GL}_2(\mathcal O_{\Sb})}

\newcommand {\OK}  {{\mathcal O_{K}}}
\newcommand {\OS}  {{\mathcal O_{S}}}
\newcommand {\OSt}  {{\mathcal O^{\times}_{S}}}
\newcommand {\Op}  {{\mathcal O_{\mathfrak p}}}

\newcommand {\OSS}  {{\mathcal O_{T}}}
\newcommand {\OSSt}  {{\mathcal O^{\times}_{T}}}

\newcommand {\QQ}  {{\mathbb Q}}
\newcommand {\ZZ}  {{\mathbb Z}}
\newcommand {\ord} {{\textnormal{ord}}}

\newcommand {\PP}  {{\mathbb P}}
\newcommand {\NN}  {{\mathbb N}}

\newcommand {\ps}  {q}
\newcommand {\ts}  {t}

\newcommand {\ta}  {s}
\newcommand {\ls}  {U}

\newcommand {\cc} {\Omega}

\newcommand {\nn} {\nu}
\newcommand {\mm} {\mu}

\newcommand {\kkk} {c}
\newcommand {\Sb} {T}

\date{}

\begin{document}
\author{Rafael von K\"anel}

\title{An effective proof of the hyperelliptic Shafarevich conjecture}


\maketitle

\begin{abstract}
Let $C$ be a hyperelliptic curve of genus $g\geq 1$ over a number field $K$ with good reduction outside a finite set of places $S$ of $K$. We prove that $C$ has a Weierstrass model over the ring of integers of $K$ with height effectively bounded only in terms of $g$, $S$ and $K$. In particular, we obtain that for any given number field $K$, finite set of places $S$ of $K$ and integer $g\geq 1$ one can in principle determine
the set of $K$-isomorphism classes of hyperelliptic curves over $K$ of genus $g$ with good reduction outside $S$.
\end{abstract}

\bigskip
\section{Introduction}\label{secintro}
Let $K$ be a number field, let $S$ be a finite set of places of $K$ and let $g\geq 1$ be an integer.  The main goal of this article is to show that for given $K$, $S$ and $g$ there is an effective constant $\cc(K,S,g)$ with the following property. If $C$ is a hyperelliptic curve over $K$ of genus $g$ with good reduction outside $S$, then there is a Weierstrass model of $C$ over $\OK$,
\begin{equation*}
Y^2=f(X),
\end{equation*}
such that the absolute logarithmic Weil height of the polynomial $f\in\OK[X]$ is at most $\cc(K,S,g)$, where $\OK$ is the ring of integers in $K$. Our results  hold in particular  for all elliptic curves over $K$ and for all smooth, projective and geometrically connected genus $2$ curves over $K,$ since they are hyperelliptic (see section \ref{secprelim}).  

We deduce a completely effective Shafarevich conjecture \cite{shafarevich:conjecture} for hyperelliptic curves over $K$ which generalizes and improves the results for elliptic curves over $K$ of Coates \cite{coates:shafarevich} and of Fuchs, W\"ustholz and the author \cite{fuvkwu:elliptic}. In addition, our effective result allows in principle to list all $K$-isomorphism classes of hyperelliptic curves over $K$ of genus $g$ with good reduction outside $S$. Such lists already exist for special $K$, $S$ and $g$. For example, the case $K=\QQ$, $S=\{2\}$ and $g=2$ was established by Merriman and Smart in \cite{mesm:genus2,smart:shafarevich}. 
We also mention that the theorem of this paper is used in the proof of \cite[Theorem 1]{rvk:szpiro} which gives an effective exponential version of Szpiro's discriminant conjecture for hyperelliptic curves over $K$.

In our proofs we combine the effective reduction theory of Evertse and Gy\H{o}ry \cite{evgy:binaryforms}, which is based on the theory of logarithmic forms, with results for Weierstrass models of hyperelliptic curves obtained by Lockhart \cite{lockhart:discriminant} and by Liu \cite{liu:models}. This method can be exhausted to deal with the corresponding problem for slightly more general curves. Generalizations are motivated by the ``effective Mordell Conjecture'', which would follow from a version of our main result for arbitrary curves over $K$ of genus at least two. However, we point out that our method certainly does not work in this  generality. To conclude the discussion of potential Diophantine applications of our results we mention  Levin's  paper \cite{levin:siegelshaf}. In this paper, Levin establishes a link ``effective Shafarevich conjecture for hyperelliptic Jacobians $\Rightarrow$ effective Siegel theorem for hyperelliptic curves''.

Par{\v{s}}in \cite{parshin:shafarevich} and Oort \cite{oort:shafarevich} proved the qualitative Shafarevich conjecture for hyperelliptic curves over $K$, and Faltings'  finiteness theorems in \cite{faltings:finiteness} cover in particular Shafarevich's conjecture for all curves over $K$ of genus at least two. However, it is not known if the proofs of Par{\v{s}}in, Oort and Faltings are effective (i.e. allow in principle to list all $K$-isomorphism classes of hyperelliptic curves over $K$ of genus $g$ with good reduction outside $S$). See for example Levin's discussion in \cite{levin:siegelshaf} or section \ref{secstat} for certain crucial aspects of the method of Par{\v{s}}in and Oort.

The plan of the paper is as follows: In section \ref{secprelim} we define the geometric objects, as for example Weierstrass models, discriminants and Weierstrass points, to state and discuss the results in section \ref{secstat}. In section \ref{secgeom} we begin with Lemma \ref{lemsb}. It allows to construct effectively principal ideal domains in $K$. Then we collect some  results for Weierstrass models to prove Proposition \ref{propgeom} which gives a Weierstrass model of $C$ with certain minimality properties. In section \ref{seceff} we go into number theory. We first give some elementary results for binary forms. Then in  Proposition \ref{propeff} we apply the theory of logarithmic forms to establish effective estimates for the height of some monic polynomials and binary forms with given discriminant. In section \ref{secproof} we prove the theorem by the following strategy. We first assume that $C$ has a $K$-rational Weierstrass point. On using results of Lockhart \cite{lockhart:discriminant} we obtain a Weierstrass model of $C$ over $\sp(\OSS)$
$$Y^2=f(X),  \ \ \ \Delta(f)\in \OSSt$$
such that the discriminant $\Delta(f)$ of $f\in \OSS[X]$ has minimality properties, for $\Sb\supseteq S$  a controlled finite set of places of $K$ and $\OSS$ the ring of $\Sb$-integers with units $\OSSt$. Then we deduce some $\Sb$-unit equations
$$x+y=1,\ \ x,y\in \OSSt.$$
Gy\H{o}ry and Yu \cite{gyyu:sunits} applied the theory of logarithmic forms to bound the height of $x,y$. On using their result we get $\tau\in \OSS$ such that $f(X+\tau)$ has height controlled by $K,S,g,\Delta(f)$ and then by $K,S,g$, since $\Delta(f)$ has minimality properties. A suitable transformation of $Y^2=f(X+\tau)$ provides the desired Weierstrass model of $C$ over $\sp(\OK)$ in this case. In the remaining case where $C$ has no $K$-rational Weierstrass point we use that hyperelliptic curves and binary forms are closely related. We slightly extend an effective theorem for binary forms of Evertse and Gy\H{o}ry \cite{evgy:binaryforms}, which is based on the theory of logarithmic forms, to combine it with a global result of the theory of Weierstrass models of Liu \cite{liu:models}. This gives a Weierstrass model of $C$ over $\sp(\OSS)$ with height bounded and then as in the first case we get the desired model over $\sp(\OK)$. 

Throughout this paper we shall use the following conventions. The number of curves over $K$ always refers to the number of $K$-isomorphism classes of these curves. We identify a closed point $\mathfrak p\in\sp(\OK)$ with the corresponding finite place $v$ of $K$ and vice versa, we denote by $\log$ the principal value of the natural logarithm and we define the maximum of the empty set and the product taken over the empty set as $1$.

\section{Geometric preliminaries}\label{secprelim}

In this section we give some definitions to state the results in the next section. Let $K$ be a number field and let $g\geq 1$ be an integer.

A hyperelliptic curve $C$ over $K$ of genus $g$ is a smooth projective and geometrically connected curve $C\rightarrow \sp (K)$ of  genus $g$
such that there is a finite morphism $\varphi:C\to \PP_K^1$  of degree 2, for $\PP_K^1$ the projective line over $K$.
For example, \cite[Proposition 7.4.9]{liu:ag} gives that all elliptic  and all smooth projective and geometrically connected genus two curves over $K$ are hyperelliptic curves over $K$. 

Let $-\textnormal{id}$ be the automorphism of a hyperelliptic curve $C$ of order two induced by a generator of the Galois group corresponding to $\varphi:C\to \PP_K^1$.
If $g\geq 2$, then a $K$-rational Weierstrass point of $C$ is a section of $C\to\sp(K)$ which is fixed by $-\textnormal{id}$ and if $g=1$, then any section of $C\rightarrow \sp(K)$ is a $K$-rational Weierstrass point of $C$.

Let $R\subseteq K$ be a Dedekind domain with field of fractions $K$. The function field $K(C)$ of $C$ takes the form $K(C)=K(X)[Y]$, where
\begin{equation}\label{eq:hyper}
Y^{2}+f_2(X)Y=f(X),\ \ f(X), f_2(X)\in R[X]
\end{equation}
and $2g+1\leq \max \bigl( 2 \deg f_2(X), \deg f(X)\bigl)\leq 2g+2$, for $\deg w(X)$ the degree of any $w(X)\in R[X]$. We say that (\ref{eq:hyper}) is a hyperelliptic equation of $C$ over $R$. The normalization $\mathcal W(f,f_2)$  of the $\sp(R)$-scheme $\sp(R[X])\cup \sp(R[1/X])$ in  $K(C)$ is called a Weierstrass model of $C$ over $\sp(R)$.
This generalizes the well-known definition of Weierstrass models of elliptic curves over $\sp(R)$, see \cite[p.442]{liu:ag} or \cite[p.42]{silverman:aoes}.
To ease notation we write $\mathcal W(f)$ for $\mathcal W(f,f_2)$ if $f_2=0$ and we define somewhat crudely the height  $h(\mathcal W(f))$ of $\mathcal W(f)$ as the absolute logarithmic height (see \cite[1.6.1]{bogu:diophantinegeometry}) of $f\in R[X]$. We define the discriminant $\Delta$ of $\mathcal W(f,f_2)$ by
\begin{equation*}
\Delta=\begin{cases}
2^{4g}\Delta(f_0) & \textnormal{for } \deg f_0=2g+2\\
2^{4g}\alpha_0^2\Delta(f_0) & \textnormal{otherwise},
\end{cases}
\end{equation*}
where $f_0=f+f_2^2/4$ has leading coefficient $\alpha_0$ and discriminant $\Delta(f_0)$. This discriminant allows to study the reductions of $C$. More precisely, let $\mathfrak p\in\sp (R)$ be closed and let $R_\mathfrak p$ be its local ring. As usual we define that the curve $C$ has good reduction at $v$ if  $C$ is the generic fiber of a smooth proper scheme over $\sp(R_{\mathfrak p})$. This is equivalent to the existence of a smooth Weierstrass model of $C$ over $\sp (R_{\mathfrak p})$. We note that a Weierstrass model of $C$ over $\sp(R_{\mathfrak p})$ with discriminant $\Delta$ is smooth if and only if $v(\Delta)=0$.
The curve $C$ has good reduction outside a set of places $S$ of $K$ if it has good reduction at all finite places $v$ of $K$ which are not in $S$.

We remark that our explicit definition of $\Delta$ is not intrinsic. However, in  \cite{rvk:szpiro} we shall show that the more sophisticated discriminant of Deligne and Saito \cite[p.155]{saito:conductor} for arithmetic surfaces over $\sp(R_{\mathfrak p})$  of generic fiber $C$ can be controlled in terms of $\Delta$.

\section{Statement of the results}\label{secstat}

In this section we state the theorem and the corollary. We also discuss several aspects of our results and methods. In the sequel $\kkk$ denotes an effective absolute constant.

To state our results we have to introduce some notation. Let $K$ be a number field, with ring of integers $\OK$. We denote by $d$ the degree of $K$ over $\QQ$ and  by $D_K$ the absolute value of the discriminant of $K$ over $\QQ$. Let $g\geq 1$ be an integer. We write $$\nn=6(2g+1)(2g)(2g-1)d^2.$$
Let $S$ be a  finite set of places  of $K$. We denote by $\OS$ the ring of $S$-integers of $K$ and by $\OSt$ its unit group. Let $h_S$ be the class number of $\OS$ and let $h_K$ be the class number of $\OK$. We remark that $h_S\leq h_K$ and we write $\lambda_S=\log_2 h_S$. To measure $S$ we take
\begin{equation}\label{eq:parameters}
\sigma=s+\lambda_S+1,  \ \ p \ \text{ and } \ N_S=\prod N(v)
\end{equation}
with the product taken over the finite places $v\in S$, for  $s$ the number of finite places in $S$, $p$ the maximum of the residue characteristics of the finite places in $S$ and $N(v)$ the number of elements in the residue field of $v$. We now can state our main result. \\

\noindent{\bf Theorem.}
\emph{There is a finite set of places $\Sb$ of $K$ containing $S$ such that if $C$ is a hyperelliptic curve over $K$ of genus $g$ with good reduction outside $S$, then there is a Weierstrass model $\mathcal W(f)$ of $C$ over $\sp(\OK)$ with discriminant $\Delta\in \OSSt$. Furthermore,}
\begin{itemize}
\item[(i)] if $C$ has a $K$-rational Weierstrass point, then $f$ is monic separable of degree $2g+1$ and $h(\mathcal W(f)) \leq (\nn\sigma)^{5\nn\sigma}N_S^{\nn/2}D_K^{\nn (\lambda_S+1)/4},$
\item[(ii)] if $C$ has no $K$-rational Weierstrass point, then $f$ is separable of degree $2g+2$ and $h(\mathcal W(f)) \leq (\nn\sigma)^{\kkk(2\nn)^3\sigma^4}p^{(3\nn)^3\sigma^4}D_K^{(3\nn)^3\sigma^4}$.
\end{itemize}

The proof shows in addition that we can take in the theorem any set of places $\Sb\supseteq S$ of $K$ with the following properties:  $\Sb$ contains at most $d\sigma$ finite places,  $\OSS$ is a principal ideal domain, the residue characteristics $\ell$ of the finite places of $\Sb$ are at most $\max(2,p,\sqrt{D_K})$ and satisfy $2\ell\in \OSSt$, and $N_T\leq (2N_SD_K^{\lambda_S/2})^d$. Such sets $\Sb$ exist and they can be determined effectively. For example if $K=\QQ$, then we can take $\Sb=S\cup\{2\}$.

The above result holds for all elliptic and all smooth projective and geometrically connected genus $2$ curves  over $K,$ since they are hyperelliptic. By adding to $\Sb$ the places of $K$ above $3$ we can assume in the elliptic case after a suitable change of variables which does not increase our bounds that $$f(X)=X^3+a_4X+a_6.$$ Therefore the theorem generalizes the results for elliptic curves over $K$ of Coates, who covered in \cite{coates:shafarevich} the case $K=\QQ$, and of \cite{fuvkwu:elliptic} to arbitrary hyperelliptic curves over $K$. Since  $N_S\leq p^{ds}$, we see that our explicit bound (take in part (i) $g=1$, $K=\QQ$) is sharper in all quantities than the explicit one of Coates. Furthermore, the  effective bounds in \cite{fuvkwu:elliptic}, which are double exponential and only explicit in terms of $S$, are reduced to fully explicit polynomial bounds (take in part (i) $g=1$). We continue the discussion of the bounds below the corollary.

From our theorem we derive a completely effective Shafarevich conjecture \cite{shafarevich:conjecture} for hyperelliptic curves over number fields. The qualitative finiteness statement was proven by Shafarevich for elliptic curves, by Par{\v{s}}in \cite{parshin:shafarevich} for genus $2$ curves and by Oort \cite{oort:shafarevich} for arbitrary hyperelliptic curves of genus $g\geq 2$, and Merriman and Smart \cite{mesm:genus2,smart:shafarevich} used \cite{evgy:binaryforms} to compile a list of all genus $2$ curves over $\QQ$ with good reduction outside $\{2\}$, see also the survey article of Poonen \cite[Section 11]{poonen:computational}. In principle, the following corollary allows to compile such lists in much more general situations. We recall that $c$ denotes an effective absolute constant.\\

\noindent{\bf Corollary.}
\emph{The $K$-isomorphism classes $H_{K,S,g}$ of hyperelliptic curves of genus $g\geq 1$ over $K$ with good reduction outside $S$ can be determined effectively and their number $N(K,S,g)$ satisfies} $$N(K,S,g)\leq\exp((\nn\sigma)^{\kkk(2\nn)^3\sigma^4}p^{(3\nn)^3\sigma^4}D_K^{(3\nn)^3\sigma^4}).\medskip$$

In particular, this is an effective finiteness theorem of Shafarevich type for hyperelliptic curves over $K$ of genus one with no $K$-rational point. 
This corollary for more general genus one curves with no $K$-rational point would give the effective finiteness of the Tate-Shafarevich group of elliptic curves over $K$. The proof of the corollary shows in addition that the number of curves in $H_{K,S,g}$ with a $K$-rational Weierstrass point is at most
$$
\exp\bigl((\nn\sigma)^{6\nn\sigma}N_S^{\nn/2}D_K^{\nn (\lambda_S+1)/4}\bigl).
$$
Let $N(K,S,E)$ be the number of elliptic curves in $H_{K,S,1}$.  
Brumer and Silverman \cite{brsi:number} proved $N(\QQ,S,E)\leq c_1N_S^{c_2}$ for absolute constants $c_1,c_2$ (see also Helfgott and Venkatesh \cite{heve:integralpoints} for an improved $c_2$). For $g\geq 2$ one can derive a similar estimate on combining the geometric ideas of this paper with  Evertse's \cite{evertse:sunits}. It leads to $$\max(N(K,S,E),N(K,S,g))\leq c_3N_S^{c_4},$$ for $c_3$ a constant depending only on $K,g$ and $c_4$ an effective constant depending only on $g$.  
We plan to give the details in a future article.

We now explain the separation of the theorem in (i) and (ii) depending on whether a $K$-rational Weierstrass point of $C$ exists or not.
In part (i) we can reduce the problem directly to solve unit equations. This has the advantage that it leads to explicit estimates in (i), which depend directly on the at the moment best bounds for unit equations. But the method of (i) only works for monic polynomials and thus can not be applied in our proof of (ii). Therein we get from Proposition \ref{propgeom} (ii) a Weierstrass model $\mathcal W(f)$ with $f$ not necessarily monic.

Next we discuss our results in view of the nice paper \cite{jore:shafarevich} of de Jong and R\'emond. Let $C$ be a smooth, projective and geometrically connected curve over $K$ of genus $g$, with good reduction outside $S$. We denote by $h_F(J)$ the absolute stable Faltings height of the Jacobian $J=\textnormal{Pic}^0(C)$ of $C$. If there exists a finite morphism $C\to \mathbb P^1_K$ of prime degree which is geometrically a cyclic cover, then \cite{jore:shafarevich} gives an effective upper bound for $h_F(J)$ in terms of $K$, $S$ and $g$.  
A motivation for such results is given by R\'emond \cite{remond:construction}.
On using Kodaira's construction, he showed that an effective upper bound for $h_F(J)$, in terms of $K$, $S$ and $g$, would imply the effective Mordell Conjecture.
To prove their result, de Jong and R\'emond generalize the method of Par{\v{s}}in \cite{parshin:shafarevich} and Oort \cite{oort:shafarevich}
which worked before only for hyperelliptic curves. They replace therein the qualitative finiteness theorem  for unit equations of Siegel-Mahler by its effective analogue of Gy\H{o}ry and Yu \cite{gyyu:sunits}. This gives an affine plane model of $C$ with bounded height, which then allows to estimate $h_F(J)$ by combining results of Bost, David and Pazuki \cite{pazuki:heights} and of R\'emond \cite{remond:rational}. On the other hand, our method relies on the effective reduction theory of Evertse and Gy\H{o}ry \cite{evgy:binaryforms} and on results for Weierstrass models of hyperelliptic curves obtained by Lockhart \cite{lockhart:discriminant} and by Liu \cite{liu:models}. 
Hence, we conclude that the method of de Jong and R\'emond and our method are rather different. We point out that the main new feature of our method is that it allows to prove the following: Any hyperelliptic curve $C$ has a  Weierstrass model, with height effectively bounded in terms of $K,S,g$, which is defined over $\OK$. This immediately implies an effective Shafarevich theorem for hyperelliptic curves in the classical sense, i.e. which allows to determine effectively the $K$-isomorphism classes in question. On the other hand, the method of Par{\v{s}}in, Oort, de Jong and R\'emond (only) gives that any hyperelliptic curve $C$ has a Weierstrass model, with  height effectively bounded in terms of $K,S,g$, which is defined over a field extension $L$ of $K$. Although $L$ can be controlled, it is not clear if this weaker statement can be used to deduce our Theorem. Further, it is not clear if the method of Par{\v{s}}in, Oort, de Jong and R\'emond allows to determine effectively the $K$-isomorphism classes of hyperelliptic curves over $K$ of genus $g$, with good reduction outside $S$;  their method (only) allows to effectively determine the smaller set of $\bar{K}$-isomorphism classes, where $\bar{K}$ is an algebraic closure of $K$.

We now consider additional aspects of our method.  Theorem (i) together with \cite[Remarque]{jore:shafarevich} shows that if $J$ is the Jacobian of a hyperelliptic curve over $K$ of genus $g$, with a $K$-rational Weierstrass point and with good reduction outside $S$, then
\begin{equation}\label{eq:faltest}
h_F(J)\leq 2^{2^{22} 9^g}(\nn\sigma)^{5\nn\sigma}N_S^{\nn/2}D_K^{\nn (\lambda_S+1)/4}.
\end{equation}
This bound depends polynomially on the  important quantity $N_S$ (see below), while the estimate of \cite{jore:shafarevich} is of the form $\exp(c(K,g)(\log N_S)^2)$ for $c(K,g)$ an effective constant depending only on $K$ and $g$. In addition, our method  gives bounds of the form (\ref{eq:faltest}) for  Jacobians of curves over $K$  which correspond to function fields $K(X)[Y]$ with a relation $Y^m=f(X)$,
for $m\geq 2$ an integer and for $f$ as in Proposition \ref{propeff} (i); see the discussion at the end of the paper. 
Moreover, on exhausting our method further generalizations should be possible in view of the general results from the effective reduction theory of Evertse and Gy\H{o}ry. However, it is clear that substantially new ideas are required to deal with arbitrary Jacobian varieties.

Finally, we explain the shape of the estimates. They depend ultimately on the theory of  logarithmic forms and we refer to Baker and W\"ustholz \cite{bawu:logarithmicforms} in which the state of the art of this theory is documented. We conducted some effort to obtain in (i) a bound which in terms of $S$ is polynomial in $\sigma^{\sigma}$ and $N_S$. This leads immediately to an estimate in terms of $S$ which is polynomial in $N_S$  and that is what we need for the applications given in \cite{rvk:szpiro}. Of course, on accepting a worse dependence on $S$ one could improve the dependence on $d$, $D_K$ or on $g$. As a consequence of calculating  the constants explicitly in every step of our proofs we obtained $$\kkk=c_6c_7,$$
for the effective absolute constants $c_6, c_7$ of \cite[Theorem 3]{evgy:binaryforms}. We note that throughout this paper the constants are calculated according to Baker's philosophy: ``Although some care has been taken to obtain numerical constants reasonably close to the best that can be acquired with the present method of proof, there is, nevertheless, little doubt that the numbers in the above inequalities can be reduced to a certain extent by means of minor refinements. In particular it will be seen that several of the numbers occurring in our estimates have been freely rounded off in order that the final conclusion should assume a simple form, and so some obvious improvements are immediately obtainable."\\

\section{Weierstrass models with minimality properties}\label{secgeom}

In this section we start with a lemma which allows to construct effectively principal ideal domains in any number field $K$. Then, after collecting some known results for hyperelliptic curves $C$ over $K$, we give Lemma \ref{lemweip}. It describes a relation between Weierstrass models  and the existence of a $K$-rational Weierstrass point of $C$. In the last part we prove  Proposition \ref{propgeom} which gives a Weierstrass model of $C$ with certain minimality properties.

Let $S$ be a finite set of places of $K$ and let $\OS$ be the ring of $S$-integers of $K$.  We write $\lambda_S=\log_2 h_S$ for $h_S$ the class number of $\mathcal O_S$. Let $\ta$,  $N_S$ and $p$ be as in (\ref{eq:parameters}) and for any finite set of places $\Sb$ of $K$ let $\ts$, $N_\Sb$ and $\ps$ be the corresponding quantities.  Let $\OK$ be the ring of integers in $K$ and let $D_{K}$ be the absolute value of the discriminant of $K$ over $\QQ$.

The next lemma allows us later to remove class group obstructions in connection with the existence of certain Weierstrass models. We thank  Sergej Gorchinskiy for improving the upper bound for $\ts$ in \cite[Lemma 4.3]{fuvkwu:elliptic} to $t\leq s+h_S-1$, and we thank the referee for further improving $t\leq s+h_S-1$ to the estimate in Lemma \ref{lemsb}.
\begin{lemma}\label{lemsb}
There exists a finite set of places $\Sb\supseteq S$ of $K$ such that $N_\Sb\leq N_S D_K^{\lambda_S/2},$ $\ps\leq \max (p,\sqrt{D_{K}}),$ $t\leq \ta+\lambda_S$ and such that $\OSS$ is a principal ideal domain.
\end{lemma}

\begin{proof}
Let $Cl(R)$ denote the class group of a Dedekind domain $R\subset K$. If $h_S=1$, then we take $T=S$. We do induction and we now assume that the statement is true for any finite set of places $S'$ of $K$ with corresponding class number $h_{S'}$ at most $h_S/2$. Let $\pi:Cl(\OK)\to Cl(\OS)$ be the canonical surjective homomorphism and for a non-trivial $\bar{\mathfrak a}\in Cl(\OS)$ we take  $\bar{\mathfrak b}\in Cl(\OK)$ with $\pi(\bar{\mathfrak b})=\bar{\mathfrak a}$. Minkowski's theorem gives a representative $\mathfrak b\subset \OK$ of $\bar{\mathfrak b}$ such that the residue field of any $\mathfrak p\in \sp(\OK)$ which divides $\mathfrak b$ has at most $\sqrt{D_K}$ elements. 
Since $\bar{\mathfrak a}$ is non-trivial and $\pi$ is a homomorphism there exists such a $\mathfrak p$ with $\pi(\bar{\mathfrak p})$ non-trivial, where $\bar{\mathfrak p}$ is the ideal class of $\mathfrak p$. 
We write $S'=S\cup \lbrace\mathfrak p\rbrace$ and then we see that the class number $h_{S'}$ of $\mathcal O_{S'}$ is at most $h_S-1$, since $\pi(\bar{\mathfrak p})$ lies in the kernel of the canonical surjective homomorphism $Cl(\OS)\rightarrow Cl(\mathcal O_{S'})$. 
Thus, on using that $h_{S'}$ divides $h_S$, we deduce that $h_{S'}\leq h_S/2$. Finally, an application of the induction hypothesis with $S'$ gives a set of places $T$ of $K$ with the desired properties.\end{proof}

Let $C$ be a hyperelliptic curve of genus $g$ defined over $K$. Let $R$ be a Dedekind domain with quotient field $K$ and with group of units $R^\times$. The following lemma is a direct consequence of \cite[Proposition 2]{liu:models}.

\begin{lemma}\label{lemliu}
Suppose that $R$ is a principal ideal domain and assume that $C$ has good reduction at all closed points in $\sp(R)$. Then there exists a Weierstrass model of $C$ over $\sp(R)$ with discriminant in $R^\times$.
\end{lemma}

Let $V^{2}=l(Z)$   and  $Y^{2}=f(X)$ be two hyperelliptic equations of $C$ over $K$ with discriminant  $\Delta'$ and $\Delta$ respectively. Then \cite[p.4581]{liu:models} gives
\begin{equation*}
\phi=\begin{pmatrix} \alpha & \beta\\ \gamma & \delta \end{pmatrix}\in \textnormal{GL}_{2}(K), \ \lambda\in K^{\times}
\end{equation*}
such that
\begin{equation*}
X=\phi Z=\frac{\alpha Z+\beta}{\gamma Z+\delta}, \ \ Y=\frac{\lambda V}{(\gamma Z+\delta)^{g+1}}
\end{equation*}
and that
\begin{equation}\label{eq:discritransform}
\Delta=\lambda^{4(2g+1)}(\det \phi)^{-2(g+1)(2g+1)}\Delta',
\end{equation}
where $\textnormal{GL}_{2}(K)$ are the invertible $2\times 2$-matrices with entries in $K$ and $\det \phi$ denotes the determinant of $\phi\in\textnormal{GL}_{2}(K)$.

Next, we discuss a relation between $K$-rational Weierstrass points and Weierstrass models of $C$ over $\sp(R)$. We are grateful to Professor Qing Liu who pointed out that the following Lemma \ref{lemweip} is well-known (see for example the results at the bottom of \cite[p.4579]{liu:models}).

\begin{lemma}\label{lemweip}
Suppose $\mathcal W(f)$ is a Weierstrass model of $C$ over $\sp(R)$. If $C$ has no $K$-rational Weierstrass point, then $f\in R[X]$ has degree $2g+2$.
\end{lemma}

Let $S$ and $\Sb$ be finite sets of places of $K$. Let $\Sigma$ be a generating system  of the free part of the units $\OSSt$ of the ring of $\Sb$-integers $\OSS$ and let $\zeta$ be a generator of the torsion part of $\OSSt$. We write $\mathcal U=(\Sigma,\zeta)$ and we say that $\epsilon\in \OSSt$ is $\mathcal U$-reduced if it takes the form $\epsilon=\zeta^{r}\prod_{\epsilon\in \Sigma}\epsilon^{r(\epsilon)}$, for integers $0\leq r,r(\epsilon)<4(g+1)(2g+1)$. Let $n_\Sb$ be the product $\prod \log N(v)$ taken over the finite places $v\in T$ for $N(v)$ the number of elements in the residue field of $v$ and let $d$ be the degree of $K$ over $\QQ$. Let $h(\alpha)$ be the absolute logarithmic height (see \cite[1.6.1]{bogu:diophantinegeometry}) of $\alpha\in K$.

\begin{proposition}\label{propgeom}
Suppose $T\supseteq S$ and $\mathcal{O}_\Sb$ is a principal ideal domain with $2\in \OSSt$. Let $\mathcal U$ be as above and let $C$ be a hyperelliptic curve over $K$ of genus $g$ with good reduction outside $S$. There is a Weierstrass model $\mathcal W(f)$ of $C$ over $\sp(\OSS)$ with $\mathcal U$-reduced discriminant $\Delta\in\OSSt$ such that
\begin{itemize}
\item[(i)] if $C$ has a $K$-rational Weierstrass point, then $f$ is separable and monic of degree $2g+1$,
\item[(ii)] if $C$ has no $K$-rational Weierstrass point, then $f$  is separable of degree $2g+2$.
\end{itemize}
Moreover, there is a $\mathcal U$ as above such that any $\mathcal U$-reduced $\Delta\in \OSSt$ satisfies $$h(\Delta)\leq(50g(\ts+d)!)^2(dD_K)^{d}n_\Sb.$$
\end{proposition}

\begin{proof}
We now take a hyperelliptic curve $C$ over $K$ of genus $g \geq 1$ with good reduction outside $S$. Since $\Sb$ contains $S$ we conclude that our curve $C$ has a fortiori good reduction outside $\Sb$.

(i) We  suppose that $C$ has a $K$-rational Weierstrass point. In a first step we construct a  Weierstrass model  of $C$ over $\sp(\OSS)$ with discriminant invertible in $\OSS$. Let $\mathfrak a\subseteq \OK$ be a representative of the Weierstrass class (see \cite[Definition 2.7]{lockhart:discriminant}) of $C$. Since $\OSS$ is a principal ideal domain we get $\alpha\in \OK$ and a fractional ideal $\mathfrak a_T$ of $\OK$ in $K$, which is composed only of primes in $\Sb$, such that
$\mathfrak a=\alpha \mathfrak a_T$.
After multiplying this ideal equation by a suitable $\Sb$-unit in $\OK$ we see that there is a representative $\mathfrak b\subseteq \OK$ of the Weierstrass class of $C$ which is composed only of primes in $\Sb$. An application of \cite[Proposition 2.8]{lockhart:discriminant} gives a Weierstrass model $\mathcal W(l,l_2)$ of $C$ over $\sp(\OK)$ with discriminant $\Delta'$ and with the following properties. The degree of $l_2$ is at most $g$, $l$ is monic of degree $2g+1$ and
$$\Delta' \OK=\mathfrak b^{4g(2g+1)}\mathfrak D_C,$$
for $\mathfrak D_{C}$ the minimal discriminant ideal of $C$ (see \cite[Definition 2.5]{lockhart:discriminant}). To see that $\Delta'\in\OSSt$ we let $\mathfrak p\in\sp(\OSS)$ be any closed point and we denote by $\mathcal O_\mathfrak p$ its local ring. Since our curve $C$ has good reduction outside $\Sb$ we get a Weierstrass model $\mathcal W$ of $C$ over $\sp(\Op)$ which is smooth. In particular this shows that the special fiber of $\mathcal W$ is smooth over the spectrum of the residue field at $\mathfrak p$. Therefore $\mathfrak p$ does not divide the discriminant of $\mathcal W_\mathfrak p$ and thus does not divide the minimal discriminant $\mathfrak D_C$ as well. Then the above representation of $\Delta' \OK$ and our choice of $\mathfrak b$ give that $\Delta'$ is not divisible by $\mathfrak p$. This proves that $\Delta'\in\OSSt$. Since $2\in \OSSt$ we see that the base change of $\mathcal W(l,l_2)$ to $\sp(\OSS)$ takes the form of a Weierstrass model $\mathcal W(l_0)$ of $C$ over $\sp(\OSS)$ with discriminant $\Delta'$, for $l_0=l+l_2^2/4$. We recall that $l$ is monic of degree $2g+1$ and that $l_2$ has degree at most $g$. This shows that $l_0$ is monic with degree $2g+1$.

In a second step we reduce the discriminant. Since $\Delta'\in \OSSt$ there exist integers $a,a(\epsilon)$ such that $\Delta'$ takes the form
$
\Delta'=\zeta^a \prod \epsilon^{a(\epsilon)}
$
with the product taken over $\epsilon\in \Sigma$. By reducing the exponents $a,a(\epsilon)$ modulo $4g(2g+1)$ we can rewrite the above equation  as $$\Delta'=\omega^{-4g(2g+1)}\zeta^r \prod_{\epsilon \in \Sigma} \epsilon^{r(\epsilon)}, \ \ 0\leq r,r(\epsilon)<4g(2g+1)$$ with $\omega\in \OSSt$ and integers $r,r(\epsilon)$. For $f(X)=\omega^{4g+2}l_0(X/\omega^{2})$ let
$\mathcal W(f)$ be the corresponding Weierstrass model of $C$ over $\sp(\OSS)$. By (\ref{eq:discritransform}) it has  $\mathcal U$-reduced discriminant $\omega^{4g(2g+1)}\Delta'\in \OSSt$ and the properties of $l_0$ imply that $f$ is monic with degree $2g+1$.

(ii) We now assume that $C$ has no $K$-rational Weierstrass point. Since our curve $C$ has good reduction outside $\Sb$ it has  good reduction at all closed points in $\sp(\OSS)$.  Then an application of Lemma \ref{lemliu} with $R=\OSS$  gives a Weierstrass model $\mathcal W(l,l_2)$ of $C$ over $\sp(\OSS)$ with discriminant $\Delta'\in \OSSt$. As in (i) we get a Weierstrass model $\mathcal W(l_0)$ of $C$ over $\sp(\OSS)$ with discriminant $\Delta'$, where $l_0=l+l_2^2/4$. Our assumption and Lemma \ref{lemweip} show that $l_0$ has degree $2g+2$. Next we reduce, in the same way as in (i), with a suitable $ \omega\in \OSSt$  the exponents of $\Delta'$ modulo $4(g+1)(2g+1)$. Let $f(X)=\omega^{4g+4}l_0(X/\omega^{2})$ and then we see that the Weierstrass model $\mathcal W(f)$ of $C$ over $\sp(\OSS)$ has the desired properties.

It remains to prove the last statement. We claim that for any finite set of places $S'$ of $K$ there is a generating system $\Sigma'$ of the free part of $\mathcal O_{S'}^{\times}$ with
\begin{equation}\label{eq:effgen}
h(\epsilon)\leq (10\left|\Sigma'\right|!)^2 (dD_K)^{d}n_{S'},  \ \ \ \epsilon\in \Sigma',
\end{equation}
for $\left|\Sigma' \right|$ the cardinality of $\Sigma'$. To see this let $R_{S'}$ be the $S'$-regulator, $R_K$ the regulator of $K$ and $h_K$ the class number of $K$. From \cite[Remark 3]{gyyu:sunits} we get that $R_{S'}\leq R_{K}h_{K}n_{S'}$ and \cite[Theorem 6.5]{lenstra:algorithms} shows that $$R_{K}h_K\leq \sqrt{D_K}\max(2d,d\log D_K)^{d-1}/(d-1)!.$$ Hence we deduce
\begin{equation}\label{eq:regulator}
R_{S'} \leq D_K^{1/2}\max(2d,d\log D_K)^{d-1}n_{S'}/(d-1)!.
\end{equation}
Then the claim follows, since \cite[Lemma 2]{gyyu:sunits} gives a generating system $\Sigma'$ of the free part of $\mathcal O_{S'}^{\times}$ such that $$h(\epsilon)\leq (10\left|\Sigma'\right|!)^2\max(1,\log d) R_{S'},\epsilon\in \Sigma'.$$
We choose  $\mathcal U=(\Sigma,\zeta)$ such that $\Sigma$ is a generating system  of the free part of $\OSSt$ with heights bounded as above.  Since $\Delta\in \OSSt$ is $\mathcal U$-reduced it takes the form $\zeta^r \prod_{\epsilon \in \Sigma} \epsilon^{r(\epsilon)}$ for integers $0\leq r, r(\epsilon)<4(g+1)(2g+1)$. Therefore the above bound for $h(\epsilon)$ together with  $\left|\Sigma\right|=\ts+d-1$ leads to the last statement. This completes the proof of the proposition.
\end{proof}

To apply this proposition in our proof  we shall extend $S$ to a set $\Sb$ such that $\OSS$ is a principal ideal domain. By Lemma \ref{lemsb} this can be done in a controlled way  with the disadvantage inherent that then $T$ depends  on $K$.

We now briefly discuss two further ideas and we let $C$ be as in the proposition above. The first idea is to use the Hilbert class field $H(K)$ of $K$. This field is  controlled in terms of $K$ and it seems that one can avoid an extension of $S$ by working with the base change of $C$ to $H(K)$ to get an equation over $H(K)$ with minimality properties.

The second idea is due to Par{\v{s}}in \cite{parshin:shafarevich}. There is a smallest finite extension $L\supseteq K$ such that the base change $C_L$ of $C$ to $L$ has $2g+2$ $L$-rational Weierstrass points. With the help of a fixed $L$-rational Weierstrass point we embed $C_L$ in its Jacobian over $L$ which extends to an abelian scheme  $\mathcal J\rightarrow\sp(\mathcal O)$, where $\mathcal O$ denotes the integral closure of $\OS[1/2]$ in $L$. Then on using that the Weierstrass points give $2$-torsion points on $\mathcal J$ we get a Weierstrass model of  $C_L$ over a localization of $\sp(\mathcal O)$. Its discriminant and also $L$ can be controlled in terms of $K,S,g$, since all places of $L$ that extend a closed point in $\sp(\OS[1/2])$ are unramified over $K$.

Probably, one can use these ideas to improve the estimates in some special cases as for example when $C$ is semi-stable over $K$. But it is clear that in general both approaches give no equation for $C$ over $K$.

\section{Binary forms and monic polynomials with given discriminant}\label{seceff}

In the first part of this section we collect some elementary results for binary forms over any number field $K$. In the second part we prove  Proposition \ref{propeff}.  It gives effective bounds for the height of certain binary  forms  and monic polynomials over $K$ of given discriminant.

Let $n\geq 1$ be an integer. Over a field extension of $K$ the binary form $$G(X,Y)=\sum_{0\leq i\leq n} \beta_iX^{n-i}Y^{i}\in K[X,Y]$$ factors as $\prod_{1\leq j\leq n}(\zeta_jX-\xi_jY)$ and we define the discriminant $\Delta(G)$ of $G$ by
\begin{equation*}
\Delta(G)=\prod_{1\leq i < j\leq n}(\zeta_i\xi_j-\zeta_j\xi_i)^2.
\end{equation*}
It has the properties (see \cite[p.169]{evgy:binaryforms}) that $\Delta(G)\in \ZZ[\beta_0,\dotsc,\beta_n]$ and that $\Delta(\alpha G)=\alpha^{2n-2}\Delta(G)$, for $\alpha\in K$. The pullback $\psi^*G$ of $G$ by $\psi \in \textnormal{GL}_2(K)$ can be written as
\begin{equation*}
\psi^*G(X,Y)=G(\alpha X+\beta Y,\gamma X+\delta Y) \textnormal{ for } \psi=\begin{pmatrix}\alpha & \beta \\ \gamma & \delta\end{pmatrix}\in \textnormal{GL}_2(K)
\end{equation*}
and has discriminant $\Delta(\psi^*G)=(\det \psi)^{n(n-1)}\Delta(G)$.

Suppose that $f(X)=\alpha_0X^n+\dotsc+\alpha_n\in K[X]$ is separable of degree $n\geq 1$. We write the monic polynomial
\begin{equation*}
\alpha_0^{-1}f(X)=\prod_{1\leq j\leq n} (X-\gamma_j)\in K[X]
\end{equation*}
as a product taken over $k$ of irreducible and monic $f_k(X)\in K[X]$ and we denote by $F$ and $F_k$ the homogenizations in $K[X,Y]$ of $f$ and $f_k$ respectively.  Let $\Sb$ be a finite set of places of  $K$ and let  $\OSS$ be the ring of $\Sb$-integers in $K$ with units $\OSSt$.  For any $A\in K[X,Y]$ we let $(A)_\Sb$ be the $\OSS$-submodule of $K$ generated by the coefficients of $A$ and if $G\neq 0$ we define $$d_{\Sb}(G)=\frac{(\Delta(G))_\Sb}{(G)_\Sb^{2n-2}}.$$ 
We now prove some elementary results for which we did not find a suitable reference.

\begin{lemma}\label{lembinarprops}
The discriminant of $F$ and of $f$ are equal. The binary forms $F_k\in K[X,Y]$ are irreducible and satisfy $\prod_k F_k=\alpha_0^{-1}F$. If $F\in \OSS[X,Y]$ is monic with $\Delta(F)\in \OSSt$, then $d_{\Sb}(F)=\OSS$.
\end{lemma}

\begin{proof}
The leading coefficient $\alpha_0\in K\setminus\{0\}$ of $F$ is the product of the elements $\zeta_j$, thus all $\zeta_j$ are nonzero and then with $F(X,1)=f(X)$ we get
\begin{equation*}
F(X,Y)=\prod_{1\leq j\leq n}\zeta_j(X-\frac{\xi_j}{\zeta_j}Y)=\alpha_0\prod_{1\leq j\leq n}(X-\gamma_jY),
\end{equation*}
which implies $\Delta(F)=\Delta(f)$. On using that $F_k(X,1)$ is irreducible in $K[X]$ we deduce that $F_k$ is irreducible in $K[X,Y]$.
We observe that $\prod_k f_k$ has the same coefficients as $\alpha_0^{-1}F$ and this implies that $\prod_k F_k=\alpha_0^{-1}F$.
The fractional ideal $(F)_\Sb^{-1}$ of $\OSS$ in $K$ consists of the elements $\alpha\in K$ such that $\alpha F\in \OSS[X,Y]$. This gives for our monic $F\in \OSS[X,Y]$ that $(F)_{\Sb}^{-1}=\OSS$ and then our assumption that $\Delta(F)\in\OSSt$ leads to $d_{\Sb}(F)=\OSS$.
\end{proof}

Next we give Proposition \ref{propeff} which allows us later to construct Weierstrass models with effectively bounded height. Its proof is based on the theory of logarithmic forms. More precisely, in part (i) we use ideas of Gy\H{o}ry and we apply a result of Gy\H{o}ry and Yu \cite{gyyu:sunits}.  Part (ii) is an application  of Gy\H{o}ry and Evertse \cite{evgy:binaryforms}. To state the proposition we have to introduce some notation. In the sequel $\kkk$ is an effective absolute constant.

For any subset $T_0$ of $\Sb$ let $\mathcal O_{T_0}$ be the ring of $T_0$-integers and let $\rho: \mathcal O_{T_0}\rightarrow \textnormal{GL}_2(K)$ be the representation given by
$$\tau\mapsto \begin{pmatrix} 1 & \tau \\ 0 & 1 \end{pmatrix}.$$
We say that $\phi\in \rho(\mathcal O_{T_0})$ is a unipotent translation and we define $\phi^*f(X)=(\phi^*F)(X,1)$. For any Dedekind domain $R\subseteq K$ we denote by $\textnormal{SL}_2(R)$  the $2\times 2$-matrices with entries in $R$ and with determinant one.

As before let $N_\Sb=\prod N(v)$ and $n_\Sb=\prod \log N(v)$ with the products taken over the finite places $v\in \Sb$ and for $N(v)$ the number of elements in the residue field of $v$. Let $\ts$ be the number of finite places in $\Sb$ and let $\ps$ be the maximum of the rational prime divisors of $N_\Sb$.

Let $d$ be the degree of $K$ over $\QQ$, let $D_K$ be  the absolute value of the discriminant of $K$ over $\QQ$ and for any $n\geq 3$ let $$\mm(n,d)=3n(n-1)(n-2)d.$$ Let $h$ be the absolute logarithmic height (see \cite[1.6.1]{bogu:diophantinegeometry}) and let $H=\exp(h)$ be the absolute multiplicative height. We now can state the following proposition.

\begin{proposition}\label{propeff}
Suppose $f\in \mathcal O_{T_0}[X]$ has degree $n\geq 3$ and discriminant $\Delta(f) \in\OSSt$. Let $F$ be the homogenization of $f$ in $\mathcal O_{T_0}[X,Y]$ and write $\mu=\mu(n,d)$.

\begin{itemize}
\item[(i)] If $f$ is monic, then there is a unipotent translation $\phi\in \textnormal{SL}_2(\mathcal O_{T_0})$ such that
$h(\phi^*f)\leq h(\Delta(f))+(N_\Sb D_K^{1/3})^{\mm} (\mm(t+1))^{4\mm(t+1)}.$
\item[(ii)] In general there is a $\phi\in\slos$ such that \\
$h(\phi^*F)\leq 32h(\Delta(F))+\ps^{2n^8d(\ts^2+1)^2}D_K^{2n^8(\ts+1)}(n(\ts+d))^{\kkk n^8d(\ts^2+1)^2}.$
\end{itemize}
\end{proposition}

\begin{proof}
In the proof we shall use well-known results in \cite[Chapter 1]{bogu:diophantinegeometry} for the height $H$ without mentioning it.

(i) We start with some notation. Since $n\geq 3$ and $\Delta(f)\neq 0$ we can choose pairwise different roots $\alpha$, $\beta$, $\gamma$ of $f$ and we write $L=K(\alpha,\beta,\gamma)$. The quantities $D_L$, $l$, $\ls$, $R_{\ls}$ and $u$ denote the absolute value of the discriminant of $L$ over $\QQ$, the degree $[L:\QQ]$, the places of $L$ which lie above $\Sb$ together with the infinite places of $L$,  the $\ls$-regulator of $L$ and the number of finite places in $\ls$ respectively. We note that $l\leq dm$ for $m=n(n-1)(n-2)$.

In a first step we show that $H(\alpha-\beta)$ is bounded explicitly in terms of $n$, $K$, $\Sb$ and $H(\Delta(f))$. The roots $\alpha,\beta,\gamma$ of our monic $f\in \mathcal O_{\ls}[X]$ are $\ls$-integral and $\Delta(f)$ is a $\ls$-unit. This shows that the factors 
$(\alpha-\beta)$, $(\beta-\gamma)$ and $(\alpha-\gamma)$ of $\Delta(f)$ are $U$-units. Therefore we get a $\ls$-unit equation
\begin{equation*}
\frac{(\alpha-\beta)}{(\alpha-\gamma)}+\frac{(\beta-\gamma)}{(\alpha-\gamma)}=1.
\end{equation*}
An application of \cite[Theorem 1]{gyyu:sunits} gives
$$\Omega_{\ls}=\exp(7\kappa_\Sb R_{\ls} N_\Sb^m\max(1,\log R_{\ls})),$$
for $\kappa_\Sb=c_1(md,m(\ts+d))$ defined in \cite[Theorem 1]{gyyu:sunits}, such that
\begin{equation}\label{eq:omega}
H(\frac{\gamma-\alpha}{\alpha-\beta})\leq \Omega_{\ls}.
\end{equation}
In what follows we shall use that $\kappa_T$ and thus $\Omega_U$ are sufficiently big, see \cite[Theorem 1]{gyyu:sunits}.
The term $\Omega_{\ls}$ depends on  $R_{\ls}$ for which we now derive an upper bound in terms of $K$, $n$ and $\Sb$. For $n_U$ defined similarly as $n_\Sb$ with $\ls$ in place of $\Sb$ we deduce $n_U\leq ((l/d)^{\ts}n_\Sb)^{l/d}$ and then (\ref{eq:regulator}), with $L$ and $U$ in place of $K$ and $S'$ respectively, leads to
\begin{equation*}
R_{\ls} \leq (2l)^{l-1}D_L^{1/2}\max(1,\log D_L)^{l-1}((l/d)^{\ts}n_\Sb)^{l/d}.
\end{equation*}
To estimate $D_L$ we first show that $\mathfrak D_{L/K}\OSS=\OSS$, where $\mathfrak D_{L/K}$ is the relative discriminant ideal of $L$ over $K$. For any fractional ideal $\mathfrak a$ of $\OK$ in $K$ let $(\mathfrak a)_\Sb=\mathfrak a\OSS$, where $\OK$ denotes the ring of integers in $K$. We consider for $\kappa\in\{\alpha,\beta,\gamma\}$ the field $M=K(\kappa)$, we write $f(X)$ as a product taken over $k$ of irreducible monic $f_k(X)\in K[X]$ and we let $F$ and $F_k$ be the homogenizations in $K[X,Y]$ of $f$ and $f_k$ respectively. Lemma \ref{lembinarprops} implies that $F=\prod_k F_k$ can be associated to a system of fields which contains the field $M$ and then \cite[Lemma 15]{evgy:binaryforms} shows $$d_{\Sb}(F)\subseteq  (\mathfrak D_{M/K})_\Sb,$$ for $\mathfrak D_{M/K}$ the relative discriminant ideal of $M$ over $K$. Thus our assumptions that $f$ is monic with $\Delta(f)\in\OSSt$ together with Lemma \ref{lembinarprops} show that $(\mathfrak D_{M/K})_\Sb$ is trivial. Let $\mathfrak d_{L/K}$ and $\mathfrak d_{M/K}$ be the relative different of $L$ and $M$ over $K$ respectively. The multiplicativity of differents in towers together with \cite[Lemma 6]{stark:effectivesiegel} leads to
$\prod_\kappa \mathfrak d_{M/K}\subseteq\mathfrak d_{L/K}$
and taking the ideal norm from $L$ into $K$ gives
$$\prod_\kappa \mathfrak D_{M/K}^{[L:M]}\subseteq \mathfrak D_{L/K}.$$
Thus we deduce that $(\mathfrak D_{L/K})_\Sb=\OSS$, since all $(\mathfrak D_{M/K})_\Sb$ are trivial. Then the arguments of \cite[p.194]{evgy:binaryforms} show
$
D_L\leq (D_K N_\Sb)^{l/d}(l/d)^{l\ts}
$
and this together with the above upper bound for $R_{\ls}$ gives $R_{\ls}\leq c_Kc_\Sb,$ for
$$c_K=D_K^{m/2}(2m^3d^2\max(1,\log D_K))^{md-1},$$
$$c_\Sb=(N_\Sb^{1/2} n_\Sb)^{m}(\max(1,\ts)m^{2\ts}\max(1, \log N_\Sb))^{md-1}.$$
Then we replace in the definition of $\Omega_{\ls}$ the term $R_{\ls}$ by $c_Kc_\Sb$ and denote by $\Omega$ the resulting term. Hence we get $\Omega_{\ls}\leq \Omega$ and since the root $\gamma\neq \alpha,\beta$ of $f$ was chosen arbitrarily it follows  from (\ref{eq:omega}) that  $$H(\Delta(f)(\alpha-\beta)^{-n(n-1)})\leq (2\Omega^2)^{n(n-1)}.$$
This leads to $H(\alpha-\beta)^{n(n-1)}\leq H(\Delta(f))(2\Omega^2)^{n(n-1)}$ which gives
\begin{equation}\label{eq:delta}
H(\alpha-\beta)\leq 2\Omega^2H(\Delta(f))^{1/(n(n-1))}.
\end{equation}

In a second step we approximate the trace $\textnormal{Tr}(f)$ of $f$ by a suitable  $\tau\in \mathcal O_{T_0}$ such that if $\alpha$ is a root of $f$, then $H(\alpha-\tau)$ is bounded explicitly in terms of $\Omega$, $n$ and $H(\Delta(f))$.  Our assumption $f\in \mathcal O_{T_0}[X]$ provides that $\textnormal{Tr}(f)\in \mathcal O_{T_0}$ and then an application of \cite[Lemma 6]{evgy:binaryforms} gives $\eta\in \OK$, $\tau\in\mathcal O_{T_0}$ such that $\eta=\textnormal{Tr}(f)-n\tau$ and that $H(\eta)\leq \Omega$.
Thus for any root $\alpha$ of $f$ we get $$n(\alpha-\tau)=\sum(\alpha-\beta)+\eta,$$ with the sum taken over the roots $\beta$ of $f$, which together with (\ref{eq:delta}) leads to
\begin{equation*}
H(\alpha-\tau)\leq \Omega^{2n}H(\Delta(f))^{1/n}.
\end{equation*}

For $\phi=\rho(\tau)$ in $\textnormal{SL}_2(\mathcal O_{T_0})$ we get $\phi^*f(X)=\prod(X-(\alpha-\tau))$ with the product taken over the roots $\alpha$ of $f$. Then we deduce from the above estimate for $h(\alpha-\tau)$, which holds for all roots $\alpha$ of $f$,  together with
$$\Omega=\exp\bigl(7 \kappa_\Sb N_\Sb^m c_Kc_\Sb \log (c_Kc_\Sb)\bigl)$$ and the definitions of $c_K,\kappa_\Sb,c_\Sb$ an upper  bound for $h(\phi^*f)$ as stated in (i). To simplify the form of the final bound we used the estimates  $m\geq 6$, $\log x\leq 3x^{1/3}$ for $x\geq 1$, $n_T\leq N_T$ and $\log (c_Kc_\Sb)\leq 3(c_Kc_\Sb)^{1/12}.$

(ii) From \cite[Theorem 3]{evgy:binaryforms} we get $\psi_0\in \slos$, $\epsilon\in \OSSt$ and effective absolute constants $c_6\geq 3$, $c_7\geq 1$ such that $H(\epsilon(\psi_0^*F))\leq \Omega$, where
\begin{equation}\label{eq:omegab}
\log\Omega=(7n)^{-2}(c_6(d+\ts)n)^{c_7dn^8(\ts+1)^2}\ps^{2dn^8(\ts+1)^2}D_K^{2n^8(\ts+1)}.
\end{equation}
We now construct with $\epsilon$ and $\psi_0$ an element $\phi\in\slos$ such that $\phi^*F$ has bounded height. From \cite[Lemma 3]{gyyu:sunits} we deduce that there exist $\Sb$-units $\epsilon_1$ and $\epsilon_2$  such that $\epsilon=\epsilon_1\epsilon_2^{-n}$ and  such that $H(\epsilon_1)$ is bounded from above by a term, which is at most $\Omega$ by \cite[Theorem 6.5]{lenstra:algorithms}.
If $\psi=\epsilon_2^{-1} \psi_0$ and $G=\psi^* F$, then we see that $G$ takes the form $\epsilon_1^{-1}\epsilon (\psi_0^* F)$ which implies
\begin{equation}\label{eq:hg}
H(G)\leq \Omega^2.
\end{equation}
For $g(X)=G(X,1)$ we get from Lemma \ref{lembinarprops} that $H(\Delta(G))=H(\Delta(g))$ and this leads to
$
H(\Delta(G))\leq 2^{3n(n-1)}H(G)^{2n-2}.
$
Then (\ref{eq:hg}) implies
$$
H(\det(\psi^{-1}))\leq \Omega^{2}H(\Delta(F))^{1/(n(n-1))},
$$
since $\det{(\psi^{-1})^{n(n-1)}}=\Delta(F)\Delta(G)^{-1}$ by $F=(\psi^{-1})^*G$.
An application of \cite[Lemma 7]{evgy:binaryforms} with the transpose of $\psi^{-1}\in \glos$ gives $\phi\in \slos$ such that the maximum $H(\psi^{-1}\phi)$ of the absolute multiplicative heights of the standard coordinates of $\psi^{-1}\phi$ is at most $\Omega H(\det(\psi^{-1}))^8$.
Thus the upper bound for $H(\det(\psi^{-1}))$ implies
\begin{equation}\label{eq:psi}
H(\psi^{-1}\phi)\leq \Omega^{17}H(\Delta(F))^{8/(n(n-1))}.
\end{equation}

In the last step  we derive an upper bound for $H(\phi^*F)$ in terms of $n$, $\Omega$ and $H(\Delta(F))$. Let $(\alpha_i)$ be the coefficients of $G$ and let
$$\psi^{-1}\phi=\begin{pmatrix} \alpha & \beta \\ \gamma & \delta \end{pmatrix}.$$
Since the pullback $^*$ induces a right-action of the group $\textnormal{GL}_2(K)$ on the set of binary forms over $K$ we get that $\phi^*F=(\psi^{-1}\phi)^*G$ which shows $$\phi^*F(X,Y)=\sum_{i=0}^{n}\alpha_i(\alpha X+\beta Y)^{n-i}(\gamma X+\delta Y)^i.$$ We deduce that
$$H(\phi^*F)\leq \ (n+1)\prod_{i=0}^{n}H(\alpha_i(\alpha X+\beta Y)^{n-i}(\gamma X+\delta Y)^i)$$
and we see that the $i$-th factor of the product on the right-hand side is at most $2^{3n} H(\alpha_i)H(\psi^{-1}\phi)^{2n}$.
Hence (\ref{eq:hg}) and (\ref{eq:psi}) give
$$H(\phi^*F)\leq \Omega^{(7n)^2}H(\Delta)^{32}$$
and then (ii) follows  from (\ref{eq:omegab}) with $\kkk=c_6c_7$. This completes the proof of the proposition.
\end{proof}

For any monic and separable $f\in K[X]$ of degree $n\geq 3$ we let $T_0$  be the smallest set of places of $K$ such that $f\in \mathcal O_{T_0}[X]$. For $T$ we take the smallest set of places of $K$ containing $T_0$ such that $\Delta(f)\in \OSSt$. Then we see that (i) gives  a unipotent translation $\phi\in \textnormal{SL}_2(\mathcal O_{T_0})$ such that
$$h(\phi^*f)\leq h(\Delta(f))+(N_T D_K^{1/3})^{\mm } (\mm(t+1))^{4\mm(t+1)},$$
for $\mu=\mu(n,d)$. The quantities $t$ and $N_T$ can be bounded effectively in terms of $\Delta(f)$ and $T_0$ such that the resulting bound improves in all parameters the actual best effective estimates (see \cite[Thm 7]{gyory:monicpolynomials}, \cite{beevgy:diophantineproblems}, \cite{gyory:update} and the references therein) and makes them completely explicit. For example we can reduce the exponent $(dn!)^2((n!d)!)$ of $N_T$ and $D_K$, which would follow from \cite[Thm 7]{gyory:monicpolynomials}, to $\mm\leq 6n^3d$. Moreover, it is shown in \cite{gyory:monicpolynomials} that an effective estimate for $h(\phi^*f)$ has several applications in algebraic number theory. These can now be stated with sharper and fully explicit bounds. Finally, we mention that Evertse and Gy\H{o}ry informed the author that they are writing a book which will  inter alia allow to improve Proposition \ref{propeff} and which will give explicitly the effective absolute constants $c_6$, $c_7$ and $c=c_6c_7$ in the above proof.

\section{Proofs}\label{secproof}

For an outline of the principal ideas of the following proof we refer to the introduction. Let $K$ be a number field of degree $d$ over $\QQ$ and denote by $\OK$ its ring of integers. Let  $D_K$ be the absolute value of the discriminant of $K$ over $\QQ$, let $S$ be a finite set of places of $K$ and let $g\geq 1$ be an integer. As before we let $h$ be the absolute logarithmic height (see \cite[1.6.1]{bogu:diophantinegeometry}).

\begin{proof}[Proof of the Theorem] We take a hyperelliptic curve $C$ of genus $g$ defined over $K$, with good reduction outside $S$, as in the theorem. First, we recall some notation. For any finite place $v$ of $K$, we denote by $N(v)$ the cardinality of the residue field of $v$.  Let $s$ be the number of finite places in $S$,  write $N_S=\prod N(v)$ with the product taken over the finite places $v\in S$ and let  $p$ be the maximum of the rational prime divisors of $N_S$; and for any finite set of places $T$  of $K$ we denote by $\ts$, $N_{\Sb}$ and $\ps$ the corresponding quantities. If $\ell$ is a rational prime, then there exist at most $d$ finite places of $K$ of residue characteristic $\ell$ and if $T$ contains all finite places of $K$ of residue characteristic $\ell$, then $\ell$ is invertible in the ring of $T$-integers $\mathcal O_T$. Thus, on using Lemma \ref{lemsb}, we see that there exists a finite set of places $\Sb$ of $K$ with the following properties. The set $\Sb$ contains the set $S$, $\OSS$ is a principal ideal domain, $2$ and all residue characteristics of the finite places in $T$ are in the group of units $\OSSt$ of $\OSS$, 
\begin{equation}\label{eq:u}
\ts\leq d(\ta+\lambda_S+1)=d\sigma, \ N_\Sb\leq (2N_S D_K^{\lambda_S/2})^{d} \ \textnormal{ and } \ps\leq \max(2,p,D_K^{1/2}).
\end{equation}
Here $\lambda_S=\log_2 h_S$ for $h_S$ the class number of $\OS$.

To prove statement (i) we can assume that $C$ has a $K$-rational Weierstrass point.  Proposition \ref{propgeom} (i) gives a  Weierstrass model $\mathcal W(l)$ of $C$ over $\sp(\OSS)$ with discriminant $\Delta'\in\OSSt$ and with the following properties. The polynomial $l$ is monic of degree $2g+1$ and its discriminant $\Delta(l)=2^{-4g}\Delta'\in \OSSt$ satisfies
$$h(\Delta(l))\leq 2(50g(\ts+d)!)^2(dD_K)^{d}N_\Sb.$$
Then we get from Proposition \ref{propeff} (i) a unipotent translation $\phi\in \slos$ such that
\begin{equation}\label{eq:fheight}
h(\phi^*l)\leq 2(\mm(\ts+1))^{4\mm (\ts+1)}(N_\Sb D_K^{1/3})^{\mm},
\end{equation}
for $\mm=3(2g+1)(2g)(2g-1)d=\nn/(2d)$.

In the next step we modify $\mathcal W(\phi^*l)$ to get a Weierstrass model of $C$ over $\sp(\OK)$ with the desired properties. To simplify notation we write
\begin{equation}\label{eq:denP}
n=2g+1,\ \ \eta=1.
\end{equation}
Let $\alpha$ be a coefficient of $\phi^*l$. We denote by $N(w)$ the cardinality of the residue field of a finite place $w$ of $\QQ(\alpha)$ and  we define $\lvert \alpha\rvert_w=N(w)^{-w(\alpha)}$, where $w(\alpha)$ is the order in $\alpha$ of the prime ideal which corresponds to $w$.  On taking the product over the finite places $w$ of $\QQ(\alpha)$ we see that
\begin{equation}\label{eq:hfin}
\delta(\alpha)=\prod \max\bigl( 1, \left|\alpha\right|_w\bigl)\in \NN
\end{equation}
is at most $H(\alpha)^d$ and that $\delta(\alpha)\alpha\in \OK$ by \cite[Lemma 4.2]{fuvkwu:elliptic}. The residue characteristic of a finite place in $\Sb$ is invertible in $\OSS$ and only the finite places $w$ of $\QQ(\alpha)$ with $\ord_{w}(\alpha)\leq -1$ contribute to the right-hand side of (\ref{eq:hfin}). Since $\phi^*l$ has coefficients in $\OSS$ this shows that
$$\omega=\prod \delta(\alpha)\in \OSSt,$$
where the product is taken over the coefficients $\alpha$ of $\phi^*l$. For $$f(X)=\omega^{2n}\phi^*l(X/\omega^2)\in\OK[X]$$ let $\mathcal W(f)$ be the corresponding Weierstrass model of $C$ over $\sp(\OK)$. By (\ref{eq:discritransform}) it has discriminant $\omega^{4(g+1-\eta)(2g+1)}\Delta'\in \OSSt$ and we see that
\begin{equation}\label{eq:wf}
h(\mathcal W(f))\leq 4dn^2h(\phi^*l).
\end{equation}
On replacing $N_{\Sb}$ and $\ts$ in (\ref{eq:fheight}) by the estimates given in (\ref{eq:u}) we conclude from (\ref{eq:wf}) that $\mathcal W(f)$ has the required properties. This completes the proof of Theorem (i).

To prove statement (ii) we can assume that $C$ has no $K$-rational Weierstrass point. Proposition \ref{propgeom} (ii) gives a Weierstrass model $\mathcal W(l)$ of $C$ over $\sp(\OSS)$ with discriminant $\Delta'\in \OSSt$ such that $l$ has degree $2g+2$ and discriminant $\Delta(l)=2^{-4g}\Delta'\in\OSSt$ which satisfies
$$h(\Delta(l))\leq 2(50g(\ts+d)!)^2(dD_K)^{d}n_\Sb.$$
Then an application of Proposition \ref{propeff} (ii) to the two variable homogenization $L$ of $l$ gives $$\phi=\begin{pmatrix}\alpha & \beta \\ \gamma & \delta \end{pmatrix}\in \slos$$
and an effective absolute constant $\kkk$ such that
\begin{equation}\label{eq:endboundF}
h(\phi^*L)\leq 2\ps^{2n^8d(\ts^2+1)^2}D_K^{2n^8(\ts+1)}(n(\ts+d))^{\kkk n^8d(\ts^2+1)^2},
\end{equation}
for $n=2g+2$.

We now use  the close relation between hyperelliptic curves and binary forms. Suppose that $\mathcal W(l)$ arises from $V^2=l(Z)$. The group $\slos$ acts on the non-constant rational functions in $K(C)$ by fractional linear transformations. We get a non-constant rational function $X=\phi^{-1} Z\in K(C)$. Therefore  $Y=V(\gamma X+\delta)^{n/2}$ is non-constant and it is in $K(C)$, since $n/2$ is an integer. Then we can rewrite $V^{2}=l(Z)$ as
\begin{equation*}
\frac{Y^2}{(\gamma X+\delta)^{n}}=l(\frac{\alpha X+\beta}{\gamma X+\delta}).
\end{equation*}
Multiplying both sides of this equation by $(\gamma X+\delta)^{n}$ gives
\begin{equation*}
Y^2=\sum_{i=0}^{n}\alpha_i(\alpha X+\beta)^{n-i}(\gamma X+\delta)^{i}=\phi^*l(X)
\end{equation*}
for $(\alpha_i)$ the coefficients of $l$ and this implies that $\mathcal W(\phi^*l)$ is a Weierstrass model of $C$ over $\sp(\OSS)$. Lemma \ref{lembinarprops} shows that $\Delta(\phi^*L)=\Delta(\phi^*l)$ and $\Delta(L)=\Delta(l)$ which gives $\Delta(\phi^*l)=\Delta(l)$, since $\phi$ is in $\slos$. Therefore $\mathcal W(\phi^*l)$ has discriminant $\Delta'\in\OSSt$. Then $h(\mathcal W(\phi^*l))=h(\phi^*L)$ and (\ref{eq:endboundF})  together with the arguments of the proof of part (i) (where now $n=2g+2$, $\eta=0$ in (\ref{eq:denP})) lead to Theorem (ii). To simplify the form of the final bound we used here the estimate $g(6dg)^2\leq \nu$. This completes the proof of the theorem.\end{proof}

It remains to prove the corollary. We denote by $N$ the number of $K$-isomorphism classes of hyperelliptic curves of genus $g$ defined over $K$ with good reduction outside $S$.
\begin{proof}[Proof of the Corollary] The theorem shows that there is an explicit constant $\Omega=\Omega(K,S,g,\kkk)$, for $\kkk$ an effective absolute constant, with the following property. Any hyperelliptic curve $C$ over $K$ of genus $g$ with good reduction outside $S$ gives a polynomial $f\in \OK[X]$ of degree at most $2g+2$ with absolute multiplicative height $H$ at most $\Omega$. If two such curves give the same $f$, then their function fields are described by the hyperelliptic equation $Y^2=f(X)$ and we see that these curves are $K$-isomorphic. This implies that $N$ is bounded from above by the number of polynomials $f\in K[X]$ of degree at most $2g+2$ that satisfy $H(f)\leq \Omega$. Thus the proof of \cite[1.6.8]{bogu:diophantinegeometry} yields
\begin{equation*}
N\leq (5\Omega)^{10d^2g}
\end{equation*}
and then (\ref{eq:fheight}) and (\ref{eq:endboundF}) lead to an upper bound for $\Omega$ which shows that the estimate of the corollary holds as stated.

The polynomials in $K[X]$ with bounded degree and bounded absolute height can be determined effectively (for details we refer to the discussions in  \cite{bogu:diophantinegeometry}). Thus the effective upper bound  given in the theorem implies that the $K$-isomorphism classes of hyperelliptic curves over $K$ of genus $g$ with good reduction  outside $S$ can be determined effectively. This completes the proof of the corollary.
\end{proof}

To conclude this article we demonstrate how our method can be used to deal with more general curves.  Let $C$ be a smooth, projective and geometrically connected curve over $K$ of genus $g$, with good reduction outside $S$. We denote by $h_F(J)$  the absolute stable Faltings height of the Jacobian $J=\textnormal{Pic}^0(C)$ of $C$. Let $L$ be a finite field extension of $K$. We denote by $D_L$ the absolute value of the discriminant of $L$ over $\QQ$ and by $d_L=[L:\QQ]$  the degree of $L$ over $\QQ$.  Let $T$ be a finite set of places of $L$. We denote by $\mathcal O_T^\times$ the group of units of the ring of $T$-integers $\mathcal O_T$ of $L$ and we define $N_T=\prod N(v)$ with the product taken over all finite places $v\in T$, where $N(v)$ is the number of elements of the residue field of $v$. 

\begin{proposition} \label{prop:extension}
There exists an effective constant $\lambda$, depending only on $D_L$, $d_L$, $N_T$ and $g$, with the following property. Assume the function field  of $C\times_K L$ takes the form $L(X)[Y]$, where $$Y^m=f(X), \ \ \ m\in\ZZ_{\geq 2},$$  for  $f\in  \mathcal O_T[X]$ a monic separable polynomial of degree at least $3$ with $\Delta(f)\in\mathcal O_T^\times$. Then it holds  $$h_F(J)\leq \lambda.$$
\end{proposition}
\begin{proof}
In the sequel $\lambda_1,\lambda_2,\dotsc$ denote effective constants depending only on $D_L$, $d_L$, $N_T$ and $g$. Let $n$ be the degree of $f$, and let $l$ be the least common multiple of $m$ and $n$. We write $l=m'm$ and $l=n'n$ with positive integers $m'$ and $n'$. The Riemann-Hurwitz formula implies 
\begin{equation}\label{eq:maxmn}
\max(m,n)\leq l\leq \lambda_1.  
\end{equation}

Let $\zeta$ be a generator of the torsion part of $\OSSt$. 
An application of (\ref{eq:effgen}) with $S'=T$ and $K=L$  gives a generating system $\Sigma=\Sigma'$ of the free part of $\OSSt$ such that 
\begin{equation}\label{eq:epsilon}
h(\epsilon)\leq \lambda_2, \ \ \ \epsilon \in \Sigma.
\end{equation}
Here we used that the cardinality $\lvert \Sigma \rvert$ of $\Sigma$, which appears in the upper bound  (\ref{eq:effgen}), satisfies $\lvert \Sigma\rvert=t+d_L-1\leq \lambda_3$, where $t$ denotes the number of finite places in $T$. Our assumption provides $\Delta=\Delta(f)\in \OSSt$. Thus there exist integers $a,a(\epsilon)$ such that $\Delta$ takes the form
$
\Delta=\zeta^a \prod_{\epsilon \in \Sigma} \epsilon^{a(\epsilon)}.
$
On reducing the exponents $a,a(\epsilon)$ modulo $l(n-1)$, we can rewrite the above equation  as $\Delta=\omega^{-l(n-1)}\zeta^r \prod_{\epsilon \in \Sigma} \epsilon^{r(\epsilon)}$ with $\omega\in \OSSt$ and integers $0\leq r,r(\epsilon)<l(n-1)$. We define $U=\omega^{n'}X$ and $V=\omega^{m'}Y$, and we write $f_\omega(U)=\omega^lf(U/\omega^{n'})\in L[U]$. It follows that
\begin{equation}\label{eq:affplanemodel}
V^m=f_\omega(U)
\end{equation}
defines an affine plane model of $C_L=C\times_K L$.
Further, we observe that $f_\omega\in \mathcal O_T[U]$ is monic,  since  $f\in \mathcal O_T[X]$ is monic and $\omega\in\mathcal O_T$, and the discriminant $\Delta_\omega$ of $f_\omega$ satisfies $\Delta_\omega=\omega^{l(n-1)}\Delta=\zeta^r \prod_{\epsilon \in \Sigma} \epsilon^{r(\epsilon)}.$
Therefore the inequalities $0\leq r(\epsilon)<l(n-1)$, (\ref{eq:maxmn}), (\ref{eq:epsilon}) and $\lvert \Sigma\rvert\leq \lambda_3$ imply  
\begin{equation}\label{eq:deltaomega}
h(\Delta_\omega)\leq \lambda_4.
\end{equation}

The monic polynomial  $f_\omega\in \mathcal O_T[U]$ is separable, with discriminant $\Delta_\omega\in \mathcal O_T^\times$. Moreover, $f_\omega$ has degree $n\geq 3$ and therefore we see that $f_\omega$ and $T$ satisfy all conditions of Proposition \ref{propeff} (i). Hence, an application of Proposition \ref{propeff} (i) with $f=f_\omega$ and $T_0=T$ gives a unipotent translation $\phi\in \textnormal{SL}_2(\mathcal O_T)$ such that $f^*=\phi^*f_\omega$ satisfies
\begin{equation}\label{eq:hf}
h(f^*)\leq \lambda_5+h(\Delta_\omega).
\end{equation}
Here we used in addition (\ref{eq:maxmn}) to estimate $n$ which appears in the upper bound of Proposition \ref{propeff} (i).
The definition of $\phi^*f_\omega$ gives $\tau\in \mathcal O_T$ such that $W=U-\tau$ satisfies $f^*(W)=f_\omega(U)$, and then (\ref{eq:affplanemodel}) shows that $V^m=f^*(W)$
defines an affine plane model of $C_L$.

Let $h_\theta(J)$ be the theta height of $J$, defined in \cite[p.760]{remond:rational}. On using R\'emond \cite[Th\'eor\`eme 1.3 and 1.5]{remond:rational}, we obtain an explicit estimate for  $h_\theta(J)$  in terms of $h(f^*)$, $m$, $n$ and $g$. 
Further, a result of Bost-David-Pazuki  implies an explicit upper bound for $h_F(J)$ in terms of $h_\theta(J)$ and $g$, 
see \cite{pazuki:heights}. Then (\ref{eq:maxmn}), (\ref{eq:deltaomega}) and (\ref{eq:hf}) lead to  $h_F(J)\leq \lambda_6$ as desired. 
\end{proof}

Proposition \ref{prop:extension} provides a tool to prove new cases of the ``effective Shafarevich conjecture'':  Suppose there exists a finite field extension $L$ of $K$ and a finite set of places $T$ of $L$ such that the function field of $C\times_K L$ satisfies the condition of Proposition \ref{prop:extension} and such that  $N_T$, $D_L$ and $d_L$ can be effectively controlled in terms of $K$, $S$ and $g$. Then Proposition \ref{prop:extension} gives an effective constant $\lambda'$, depending only on $K$, $S$ and $g$, such that  $h_F(J)\leq \lambda'.$

\newpage

{\bf Acknowledgements:}
It is a pleasure to thank Gisbert W\"ustholz and Clemens Fuchs for their excellent support during my PhD studies and for their helpful comments which improved the exposition of the above text.
Many thanks go to Philipp Habegger, Sergey Gorchinskiy and Professor Alexei Par{\v{s}}in for answering several questions, to Professor  Bjorn Poonen for his suggestion to look at \cite{mesm:genus2,smart:shafarevich}, to Professor Qing Liu for an interesting remark, and to the unknown referee for useful comments which clarified certain discussions and for a comment which allowed to replace in the upper bounds $h_S-1$ by $\log_2 h_S$. In addition, the author would like to thank Gisbert W\"ustholz for financial support during spring semester 2011.
This material is based upon work supported by the National Science Foundation under agreement No. DMS-0635607.

{\scriptsize
\bibliographystyle{amsalpha}
\bibliography{../../literature}
}

\noindent IH\'ES, 35 Route de Chartres, 91440 Bures-sur-Yvette, France\\
E-mail adress: {\sf rvk@ihes.fr}
\end{document}